\documentclass[11pt]{article}
\usepackage[margin=1.0in]{geometry}

\usepackage{verbatim}
\usepackage{hyperref}
\usepackage{amsmath}
\usepackage{amsthm}
\usepackage{ulem}
\usepackage{amssymb}
\usepackage{amsfonts}
\usepackage{multicol}
\usepackage{subfigure}
\usepackage[numbers,sort&compress]{natbib}
\usepackage{booktabs}
\usepackage{titlesec}
\usepackage{longtable,tabu}

\usepackage{enumitem}
\setlist[enumerate]{label={\arabic*.}}

\usepackage{color}

\newcommand{\infcrit}[2]{G(#1,#2)}

\usepackage{tikz}
\usetikzlibrary{calc}

\usepackage{tkz-graph}

\newtheorem{theorem}{Theorem}[section]
\newtheorem{lemma}[theorem]{Lemma}

\newtheorem{corollary}[theorem]{Corollary}

\theoremstyle{definition}
\newtheorem{definition}[theorem]{Definition}

\newcommand{\squishlist}{
 \begin{list}{$\bullet$}
  { \setlength{\itemsep}{0pt}
     \setlength{\parsep}{3pt}
     \setlength{\topsep}{3pt}
     \setlength{\partopsep}{0pt}
     \setlength{\leftmargin}{2.5em}
     \setlength{\labelwidth}{1em}
     \setlength{\labelsep}{0.5em} } }

\newcommand{\squishlisttwo}{
 \begin{list}{$\triangleright$}
  { \setlength{\itemsep}{0pt}
     \setlength{\parsep}{0pt}
    \setlength{\topsep}{0pt}
    \setlength{\partopsep}{0pt}
    \setlength{\leftmargin}{2em}
    \setlength{\labelwidth}{1.5em}
    \setlength{\labelsep}{0.5em} } }

\newcommand{\squishend}{
  \end{list}  }

\begin{document}

\title{Infinite families of $k$-vertex-critical ($P_5$, $C_5$)-free graphs}
\author{
Ben Cameron\\
\and
Ch\'{i}nh T. Ho\`{a}ng\\
}

\date{\today}

\maketitle

\begin{abstract}

A graph is $k$-vertex-critical if $\chi(G)=k$ but $\chi(G-v)<k$ for all $v\in V(G)$. We construct a new infinite families of $k$-vertex-critical $(P_5,C_5)$-free graphs for all $k\ge 6$. Our construction generalizes known constructions for $4$-vertex-critical $P_7$-free graphs and $5$-vertex-critical $P_5$-free graphs and is in contrast to the fact that there are only finitely many $5$-vertex-critical $(P_5,C_5)$-free graphs. In fact, our construction is actually even more well-structured, being $(2P_2,K_3+P_1,C_5)$-free.
\end{abstract}

\section{Introduction}
For a given integer $k$, the $k$-\textsc{Colouring} problem is to determine whether a given graph is $k$-colourable. This problem is known to be NP-complete for all $k\ge 3$~\cite{Karp1972} which led to tremendous interest in restricting the structure of input graphs such that polynomial-time $k$-\textsc{Colouring} algorithms can be developed. One of the most popular structural restrictions is to forbid an induced subgraph. Yet, for every graph $H$ that contains an induced cycle~\cite{KaminskiLozin2007} or claw~\cite{LevenGail1983, Holyer1981}, it remains NP-complete to solve $k$-\textsc{Colouring} on $H$-free graphs for all $k\ge 3$. Thus, assuming P$\neq$NP, if $k$-\textsc{Colouring} can be solved in polynomial time for $H$-free graphs for $k\ge 3$, $H$ must be a \textit{linear forest}, that is, a disjoint union of paths. 

To this end, it was shown that $k$-\textsc{Colouring} can be solved in polynomial-time for $P_5$-free graphs for all $k$~\cite{Hoang2010}. Further, $4$-\textsc{Colouring} $P_6$-free graphs~\cite{P6free1,P6free2,P6freeconf} and $3$-\textsc{Colouring} $P_7$-free graphs~\cite{Bonomo2018} can also be solved in polynomial-time. However,  solving $k$-\textsc{Colouring} of $P_t$-free graphs remains NP-complete if $t\ge 7$ and $k\ge 4$ or $t=6$ and $k\ge 5$~\cite{Huang2016}.  The complexity remains an open question for $P_t$-free when $t\ge 8$.

The fact that $P_5$ is the largest connected subgraph that can be forbidden such that $k$-\textsc{Colouring} can be solved in polynomial-time for all $k$ (again assuming P$\neq$NP) has generated special interest in further properties of $P_5$-free graphs. One such property is determining which subfamilies of $P_5$-free graphs admit polynomial-time \textit{certifying} $k$-\textsc{Colouring} algorithms. An algorithm is certifying if it returns a simple and easily verifiable witness with each output. The $k$-\textsc{Colouring} algorithms for $P_5$-free graphs developed in~\cite{Hoang2010} return $k$-colourings of the input graph in the case that the graph is $k$-colourable, but there is no certificate when the graph is not $k$-colourable. In general, a $k$-\textsc{Colouring} algorithm can return a $(k+1)$-vertex-critical induced subgraph to certify negative outputs. In fact, when the set of $(k+1)$-vertex-critical graphs in a given family of graphs is finite, then there exists a polynomial-time certifying $k$-\textsc{Coloring} algorithm for that family (see, for example,~\cite{P5banner2019}).

For $k$-vertex-critical $P_5$-free graphs, it is known that there are only finitely many for $k=4$~\cite{Hoang2015} and this finite list was used to develop a \textit{linear-time} $3$-\textsc{Colouring} algorithm for $P_5$-free graphs~\cite{MaffrayGregory2012}.  For $k\ge 5$, however, there are infinitely many $k$-vertex-critical $P_5$-free graphs~\cite{Hoang2015}. Thus, the search for subfamilies of $P_5$-free graphs admitting polynomial-time certifying $k$-\textsc{Coloruing} algorithms for $k\ge 4$ has turned to graphs that are $P_5$-free and $H$-free for other graphs $H$. We call these graphs $(P_5,H)$-free. The most comprehensive result to date on $k$-vertex-critical $(P_5,H)$-free graphs is the dichotomy theorem from~\cite{KCameron2021} that there are only finitely many for all $k$ if and only if $H$ is not $2P_2$ or $K_3+P_1$. It is further known that there are only finitely many $5$-vertex-critical $(P_5,H)$-free graphs when $H$ is any of:
\begin{itemize}
\item chair~\cite{HuangLi2023};
\item bull~\cite{HuangLiXia2023};
\item $C_5$~\cite{Hoang2015}.
\end{itemize}

\noindent For the above graphs, the cardinality of $k$-vertex-critical graphs remains unknown for all $k\ge 6$. 

On the infinite side, there are only two known constructions of $H$-free graphs where $H$ is a linear forest:

\begin{itemize}
\item $G_r$ is $4$-vertex-critical $P_7$-free for all $r\ge 1$~\cite{Chudnovsky4criticalconnected2020}, and
\item $G_p$ is $5$-vertex-critical $P_5$-free for all $p\ge 1$ (in fact, $(2P_2,K_3+P_1)$-free)~\cite{Hoang2015}.
\end{itemize}

These two families turn out to be special cases of the new infinite families of $k$-vertex-critical $(P_5,C_5)$-free graphs for each $k\ge 6$ that we will construct in this paper. Our families are the first defined by forbidden induced subgraphs where the boundary between only finitely many and infinitely many $k$-vertex-critical graphs is $k=6$ (all others happen at $k\le 5$). In light of the Strong Perfect Graph Theorem~\cite{Chudnovsky2006}, our result is somewhat surprising as every non-complete $P_5$-free $k$-vertex-critical graph must contain $C_5$ or $\overline{C_{2k+1}}$ for some $k\ge 3$. In addition, our result makes progress toward an open problem posed in~\cite{KCameron2021} to find a dichotomy theorem for the number of $k$-vertex-critical $(P_5,H)$-free graphs for all $k$ when $H$ is of order $5$. To see the state-of-the-art on this problem, see Table~\ref{tab:order5suvery}.

We provide our infinite families of vertex-critical graphs in Section~\ref{sec:infcrit}. Following that we provide a large table summarizing whether there are only finitely many or infinitely many $k$-vertex-critical $(P_5,H)$-free graphs for all $k$ in Section~\ref{sec:stateoftheart}. From this, we find that there are only $12$ such graphs $H$ remaining to get the complete dichotomy to solve the open problem from~\cite{KCameron2021}. We end this section with a brief section on notations and definitions.

\subsection{Notation and Definitions}

Let $\chi(G)$ denote the \textit{chromatic number} of $G$. A graph $G$ is \textit{$k$-vertex-critical} if $\chi(G)=k$ and $\chi(G-v)<k$ for all $v\in V(G)$. For $S\subseteq V(G)$, let $G[S]$ denote the subgraph of $G$ induced by $S$. For vertices $u,v\in V(G)$, we write $u\sim v$ if $u$ and $v$ are adjacent and $u\nsim v$ is $u$ and $v$ are non-adjacent. For $S,T\subseteq V(G)$, we say $S$ is \textit{complete} (\textit{anti-complete}) to $T$ if $s\sim t$ ($s\nsim t$) for all $s\in S$ and $t\in T$. For a $v\in V(G)$, $N(v)$ is the \textit{open neighbourhood} and denotes the set $\{u\in V(G): u\sim v\}$ and $N[v]$ is the \textit{closed neighbourhood}, denoting the set $N(v)\cup\{v\}$. 
For $S\subseteq V(G)$, we say $S$ is a \textit{stable set} if $u\nsim v$ for all $u,v\in S$. For graphs $G$ and $H$, $G+H$ denotes the \textit{disjoint union} of $G$ and $H$ where $V(G+H)=V(G)\cup V(H)$ and $E(G+H)=E(G)\cup E(H)$. A graph $G$ is $H$-free if it does not contain any induced subgraph isomorphic to $H$.

%
%

\section{New infinite families of $k$-vertex-critical graphs}\label{sec:infcrit}
We are now ready to give the constructions for our infinite families of vertex-critical graphs. 
\begin{definition}
Fix $q\ge 1$ and $k\ge 3$. Let $\infcrit{q}{k}$ be a graph on vertex set $\{v_0,v_1,...,v_{kq}\}$ where, with each integer taken modulo $kq+1$, the neighbourhood of vertex $v_i$ is $$\{v_{i-1},v_{i+1}\}\cup\{v_{i+kj+m} : m=2,3,...k-1\text{ and } j=0,1,...,q-1\}.$$ 
\end{definition}

Examples to help visualize this definition can be found in Figure~\ref{fig:G36}
\setcounter{subfigure}{0}
\begin{center}
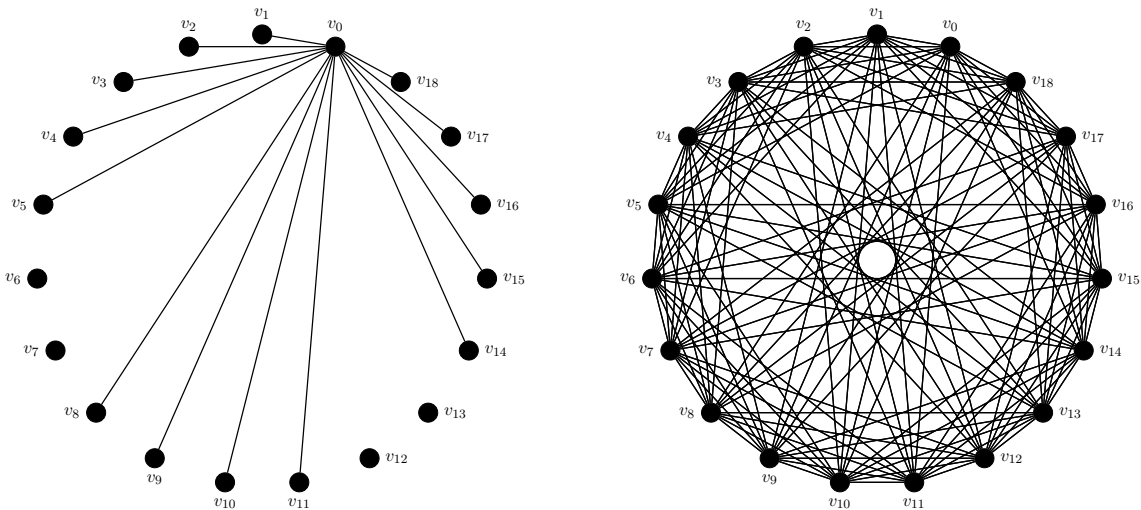
\begin{figure}[h]
\def\c{0.6}
\centering
\scalebox{\c}{
\subfigure{
\begin{tikzpicture}
\def\r{5}
\GraphInit[vstyle=Classic]
\Vertex[Lpos=90,L=\hbox{$v_0$},x=\r*0.32469946920468357cm,y=\r*0.9458172417006346cm]{v0}
\Vertex[Lpos=90,L=\hbox{$v_1$},x=\r*0.0cm,y=\r*1.0cm]{v1}
\Vertex[Lpos=90,L=\hbox{$v_2$},x=\r*-0.32469946920468357cm,y=\r*0.9458172417006346cm]{v2}
\Vertex[Lpos=180,L=\hbox{$v_3$},x=\r*-0.6142127126896678cm,y=\r*0.7891405093963936cm]{v3}
\Vertex[Lpos=180,L=\hbox{$v_4$},x=\r*-0.8371664782625287cm,y=\r*0.5469481581224268cm]{v4}
\Vertex[Lpos=180,L=\hbox{$v_5$},x=\r*-0.9694002659393304cm,y=\r*0.24548548714079912cm]{v5}
\Vertex[Lpos=180,L=\hbox{$v_6$},x=\r*-0.9965844930066698cm,y=\r*-0.08257934547233232cm]{v6}
\Vertex[Lpos=180,L=\hbox{$v_7$},x=\r*-0.9157733266550574cm,y=\r*-0.4016954246529694cm]{v7}
\Vertex[Lpos=180,L=\hbox{$v_8$},x=\r*-0.7357239106731317cm,y=\r*-0.6772815716257411cm]{v8}
\Vertex[Lpos=270,L=\hbox{$v_9$},x=\r*-0.4759473930370736cm,y=\r*-0.879473751206489cm]{v9}
\Vertex[Lpos=270,L=\hbox{$v_{10}$},x=\r*-0.16459459028073398cm,y=\r*-0.9863613034027223cm]{v10}
\Vertex[Lpos=270,L=\hbox{$v_{11}$},x=\r*0.16459459028073398cm,y=\r*-0.9863613034027223cm]{v11}
\Vertex[L=\hbox{$v_{12}$},x=\r*0.4759473930370736cm,y=\r*-0.879473751206489cm]{v12}
\Vertex[L=\hbox{$v_{13}$},x=\r*0.7357239106731317cm,y=\r*-0.6772815716257411cm]{v13}
\Vertex[L=\hbox{$v_{14}$},x=\r*0.9157733266550574cm,y=\r*-0.4016954246529694cm]{v14}
\Vertex[L=\hbox{$v_{15}$},x=\r*0.9965844930066698cm,y=\r*-0.08257934547233232cm]{v15}
\Vertex[L=\hbox{$v_{16}$},x=\r*0.9694002659393304cm,y=\r*0.24548548714079912cm]{v16}
\Vertex[L=\hbox{$v_{17}$},x=\r*0.8371664782625287cm,y=\r*0.5469481581224268cm]{v17}
\Vertex[L=\hbox{$v_{18}$},x=\r*0.6142127126896678cm,y=\r*0.7891405093963936cm]{v18}

\Edge[](v0)(v1)
\Edge[](v0)(v18)
\Edge[](v0)(v2)
\Edge[](v0)(v3)
\Edge[](v0)(v4)
\Edge[](v0)(v5)
\Edge[](v0)(v8)
\Edge[](v0)(v9)
\Edge[](v0)(v10)
\Edge[](v0)(v11)
\Edge[](v0)(v14)
\Edge[](v0)(v15)
\Edge[](v0)(v16)
\Edge[](v0)(v17)

%

\end{tikzpicture}}}
\qquad
\subfigure{
\scalebox{\c}{
\begin{tikzpicture}
\def\r{5}
\GraphInit[vstyle=Classic]
\Vertex[Lpos=90,L=\hbox{$v_0$},x=\r*0.32469946920468357cm,y=\r*0.9458172417006346cm]{v0}
\Vertex[Lpos=90,L=\hbox{$v_1$},x=\r*0.0cm,y=\r*1.0cm]{v1}
\Vertex[Lpos=90,L=\hbox{$v_2$},x=\r*-0.32469946920468357cm,y=\r*0.9458172417006346cm]{v2}
\Vertex[Lpos=180,L=\hbox{$v_3$},x=\r*-0.6142127126896678cm,y=\r*0.7891405093963936cm]{v3}
\Vertex[Lpos=180,L=\hbox{$v_4$},x=\r*-0.8371664782625287cm,y=\r*0.5469481581224268cm]{v4}
\Vertex[Lpos=180,L=\hbox{$v_5$},x=\r*-0.9694002659393304cm,y=\r*0.24548548714079912cm]{v5}
\Vertex[Lpos=180,L=\hbox{$v_6$},x=\r*-0.9965844930066698cm,y=\r*-0.08257934547233232cm]{v6}
\Vertex[Lpos=180,L=\hbox{$v_7$},x=\r*-0.9157733266550574cm,y=\r*-0.4016954246529694cm]{v7}
\Vertex[Lpos=180,L=\hbox{$v_8$},x=\r*-0.7357239106731317cm,y=\r*-0.6772815716257411cm]{v8}
\Vertex[Lpos=270,L=\hbox{$v_9$},x=\r*-0.4759473930370736cm,y=\r*-0.879473751206489cm]{v9}
\Vertex[Lpos=270,L=\hbox{$v_{10}$},x=\r*-0.16459459028073398cm,y=\r*-0.9863613034027223cm]{v10}
\Vertex[Lpos=270,L=\hbox{$v_{11}$},x=\r*0.16459459028073398cm,y=\r*-0.9863613034027223cm]{v11}
\Vertex[L=\hbox{$v_{12}$},x=\r*0.4759473930370736cm,y=\r*-0.879473751206489cm]{v12}
\Vertex[L=\hbox{$v_{13}$},x=\r*0.7357239106731317cm,y=\r*-0.6772815716257411cm]{v13}
\Vertex[L=\hbox{$v_{14}$},x=\r*0.9157733266550574cm,y=\r*-0.4016954246529694cm]{v14}
\Vertex[L=\hbox{$v_{15}$},x=\r*0.9965844930066698cm,y=\r*-0.08257934547233232cm]{v15}
\Vertex[L=\hbox{$v_{16}$},x=\r*0.9694002659393304cm,y=\r*0.24548548714079912cm]{v16}
\Vertex[L=\hbox{$v_{17}$},x=\r*0.8371664782625287cm,y=\r*0.5469481581224268cm]{v17}
\Vertex[L=\hbox{$v_{18}$},x=\r*0.6142127126896678cm,y=\r*0.7891405093963936cm]{v18}
\begin{scope}
\Edge[](v0)(v1)
\Edge[](v0)(v18)
\Edge[](v0)(v2)
\Edge[](v0)(v3)
\Edge[](v0)(v4)
\Edge[](v0)(v5)
\Edge[](v0)(v8)
\Edge[](v0)(v9)
\Edge[](v0)(v10)
\Edge[](v0)(v11)
\Edge[](v0)(v14)
\Edge[](v0)(v15)
\Edge[](v0)(v16)
\Edge[](v0)(v17)
\Edge[](v1)(v2)
\Edge[](v1)(v0)
\Edge[](v1)(v3)
\Edge[](v1)(v4)
\Edge[](v1)(v5)
\Edge[](v1)(v6)
\Edge[](v1)(v9)
\Edge[](v1)(v10)
\Edge[](v1)(v11)
\Edge[](v1)(v12)
\Edge[](v1)(v15)
\Edge[](v1)(v16)
\Edge[](v1)(v17)
\Edge[](v1)(v18)
\Edge[](v2)(v3)
\Edge[](v2)(v1)
\Edge[](v2)(v4)
\Edge[](v2)(v5)
\Edge[](v2)(v6)
\Edge[](v2)(v7)
\Edge[](v2)(v10)
\Edge[](v2)(v11)
\Edge[](v2)(v12)
\Edge[](v2)(v13)
\Edge[](v2)(v16)
\Edge[](v2)(v17)
\Edge[](v2)(v18)
\Edge[](v2)(v0)
\Edge[](v3)(v4)
\Edge[](v3)(v2)
\Edge[](v3)(v5)
\Edge[](v3)(v6)
\Edge[](v3)(v7)
\Edge[](v3)(v8)
\Edge[](v3)(v11)
\Edge[](v3)(v12)
\Edge[](v3)(v13)
\Edge[](v3)(v14)
\Edge[](v3)(v17)
\Edge[](v3)(v18)
\Edge[](v3)(v0)
\Edge[](v3)(v1)
\Edge[](v4)(v5)
\Edge[](v4)(v3)
\Edge[](v4)(v6)
\Edge[](v4)(v7)
\Edge[](v4)(v8)
\Edge[](v4)(v9)
\Edge[](v4)(v12)
\Edge[](v4)(v13)
\Edge[](v4)(v14)
\Edge[](v4)(v15)
\Edge[](v4)(v18)
\Edge[](v4)(v0)
\Edge[](v4)(v1)
\Edge[](v4)(v2)
\Edge[](v5)(v6)
\Edge[](v5)(v4)
\Edge[](v5)(v7)
\Edge[](v5)(v8)
\Edge[](v5)(v9)
\Edge[](v5)(v10)
\Edge[](v5)(v13)
\Edge[](v5)(v14)
\Edge[](v5)(v15)
\Edge[](v5)(v16)
\Edge[](v5)(v0)
\Edge[](v5)(v1)
\Edge[](v5)(v2)
\Edge[](v5)(v3)
\Edge[](v6)(v7)
\Edge[](v6)(v5)
\Edge[](v6)(v8)
\Edge[](v6)(v9)
\Edge[](v6)(v10)
\Edge[](v6)(v11)
\Edge[](v6)(v14)
\Edge[](v6)(v15)
\Edge[](v6)(v16)
\Edge[](v6)(v17)
\Edge[](v6)(v1)
\Edge[](v6)(v2)
\Edge[](v6)(v3)
\Edge[](v6)(v4)
\Edge[](v7)(v8)
\Edge[](v7)(v6)
\Edge[](v7)(v9)
\Edge[](v7)(v10)
\Edge[](v7)(v11)
\Edge[](v7)(v12)
\Edge[](v7)(v15)
\Edge[](v7)(v16)
\Edge[](v7)(v17)
\Edge[](v7)(v18)
\Edge[](v7)(v2)
\Edge[](v7)(v3)
\Edge[](v7)(v4)
\Edge[](v7)(v5)
\Edge[](v8)(v9)
\Edge[](v8)(v7)
\Edge[](v8)(v10)
\Edge[](v8)(v11)
\Edge[](v8)(v12)
\Edge[](v8)(v13)
\Edge[](v8)(v16)
\Edge[](v8)(v17)
\Edge[](v8)(v18)
\Edge[](v8)(v0)
\Edge[](v8)(v3)
\Edge[](v8)(v4)
\Edge[](v8)(v5)
\Edge[](v8)(v6)
\Edge[](v9)(v10)
\Edge[](v9)(v8)
\Edge[](v9)(v11)
\Edge[](v9)(v12)
\Edge[](v9)(v13)
\Edge[](v9)(v14)
\Edge[](v9)(v17)
\Edge[](v9)(v18)
\Edge[](v9)(v0)
\Edge[](v9)(v1)
\Edge[](v9)(v4)
\Edge[](v9)(v5)
\Edge[](v9)(v6)
\Edge[](v9)(v7)
\Edge[](v10)(v11)
\Edge[](v10)(v9)
\Edge[](v10)(v12)
\Edge[](v10)(v13)
\Edge[](v10)(v14)
\Edge[](v10)(v15)
\Edge[](v10)(v18)
\Edge[](v10)(v0)
\Edge[](v10)(v1)
\Edge[](v10)(v2)
\Edge[](v10)(v5)
\Edge[](v10)(v6)
\Edge[](v10)(v7)
\Edge[](v10)(v8)
\Edge[](v11)(v12)
\Edge[](v11)(v10)
\Edge[](v11)(v13)
\Edge[](v11)(v14)
\Edge[](v11)(v15)
\Edge[](v11)(v16)
\Edge[](v11)(v0)
\Edge[](v11)(v1)
\Edge[](v11)(v2)
\Edge[](v11)(v3)
\Edge[](v11)(v6)
\Edge[](v11)(v7)
\Edge[](v11)(v8)
\Edge[](v11)(v9)
\Edge[](v12)(v13)
\Edge[](v12)(v11)
\Edge[](v12)(v14)
\Edge[](v12)(v15)
\Edge[](v12)(v16)
\Edge[](v12)(v17)
\Edge[](v12)(v1)
\Edge[](v12)(v2)
\Edge[](v12)(v3)
\Edge[](v12)(v4)
\Edge[](v12)(v7)
\Edge[](v12)(v8)
\Edge[](v12)(v9)
\Edge[](v12)(v10)
\Edge[](v13)(v14)
\Edge[](v13)(v12)
\Edge[](v13)(v15)
\Edge[](v13)(v16)
\Edge[](v13)(v17)
\Edge[](v13)(v18)
\Edge[](v13)(v2)
\Edge[](v13)(v3)
\Edge[](v13)(v4)
\Edge[](v13)(v5)
\Edge[](v13)(v8)
\Edge[](v13)(v9)
\Edge[](v13)(v10)
\Edge[](v13)(v11)
\Edge[](v14)(v15)
\Edge[](v14)(v13)
\Edge[](v14)(v16)
\Edge[](v14)(v17)
\Edge[](v14)(v18)
\Edge[](v14)(v0)
\Edge[](v14)(v3)
\Edge[](v14)(v4)
\Edge[](v14)(v5)
\Edge[](v14)(v6)
\Edge[](v14)(v9)
\Edge[](v14)(v10)
\Edge[](v14)(v11)
\Edge[](v14)(v12)
\Edge[](v15)(v16)
\Edge[](v15)(v14)
\Edge[](v15)(v17)
\Edge[](v15)(v18)
\Edge[](v15)(v0)
\Edge[](v15)(v1)
\Edge[](v15)(v4)
\Edge[](v15)(v5)
\Edge[](v15)(v6)
\Edge[](v15)(v7)
\Edge[](v15)(v10)
\Edge[](v15)(v11)
\Edge[](v15)(v12)
\Edge[](v15)(v13)
\Edge[](v16)(v17)
\Edge[](v16)(v15)
\Edge[](v16)(v18)
\Edge[](v16)(v0)
\Edge[](v16)(v1)
\Edge[](v16)(v2)
\Edge[](v16)(v5)
\Edge[](v16)(v6)
\Edge[](v16)(v7)
\Edge[](v16)(v8)
\Edge[](v16)(v11)
\Edge[](v16)(v12)
\Edge[](v16)(v13)
\Edge[](v16)(v14)
\Edge[](v17)(v18)
\Edge[](v17)(v16)
\Edge[](v17)(v0)
\Edge[](v17)(v1)
\Edge[](v17)(v2)
\Edge[](v17)(v3)
\Edge[](v17)(v6)
\Edge[](v17)(v7)
\Edge[](v17)(v8)
\Edge[](v17)(v9)
\Edge[](v17)(v12)
\Edge[](v17)(v13)
\Edge[](v17)(v14)
\Edge[](v17)(v15)
\Edge[](v18)(v0)
\Edge[](v18)(v17)
\Edge[](v18)(v1)
\Edge[](v18)(v2)
\Edge[](v18)(v3)
\Edge[](v18)(v4)
\Edge[](v18)(v7)
\Edge[](v18)(v8)
\Edge[](v18)(v9)
\Edge[](v18)(v10)
\Edge[](v18)(v13)
\Edge[](v18)(v14)
\Edge[](v18)(v15)
\Edge[](v18)(v16)

\end{scope}
\end{tikzpicture}}
}
\caption{The neighbourhood of $v_0$ in $G(3,6)$ (left) and the graph $G(3,6)$ (right).}\label{fig:G36}
\end{figure}
\end{center}

From this definition, it is clear that $\{\infcrit{q}{3}:q\ge 1\}=\{G_r:r\ge 1\}$ where $G_r$ is the infinite family $4$-vertex-critical $P_7$-free graph in~\cite{Chudnovsky4criticalconnected2020} and $\{\infcrit{q}{4}:q\ge 1\}=\{G_p:p\ge 1\}$ where $G_p$ is the infinite family of $5$-vertex-critical $2P_2$-free graphs in~\cite{Hoang2015}. Further, our new construction is very natural in the sense that $\infcrit{1}{k}\cong K_{k+1}$ and $\infcrit{2}{k}\cong\overline{C_{2k+1}}$ for all $k$, which include all prototypical $(k+1)$-vertex-critical $P_5$-free graphs except $C_5$ from the Strong Perfect Graph Theorem~\cite{Chudnovsky2006}.

For a given $q$ and $k$ and for $0\le i\le k-1$, let $V_i=\{v_t:t=\ i\pmod{k}\}$. It is clear that the $V_i$'s partition the vertex set of $\infcrit{q}{k}$. 

\begin{lemma}\label{lem:VisstablesetsinGqkexceptV0}
For $1\le i\le k$, $V_i$ is a stable set of $\infcrit{q}{k}$ and the only edge in $V_0$ is $v_0v_{qk}$.
\end{lemma}
\begin{proof}
Fix $i\in\{0,1,2\ldots, qk\}$. Since it is clear that $v_i\nsim v_{i+1}$ and $v_{i}\nsim v_{i-1}$ with the exception of $v_0\sim v_{qk}$, we will focus on the subset of $N(v_i)$ defined by $\{v_{i+kj+m} : m=2,3,...k-1\text{ and } j=0,1,...,q-1\}$. Fix $m\in\{2,3,...k-1\}$ and $j=\{0,1,...,q-1\}$. The result now follows by considering  two cases to consider.\\

\noindent\textbf{Case 1:} Suppose $i+kj+m\le qk$. Then $i+kj+m\equiv i+m\pmod{k}$, and since $2\le m \le k-1$, $i\not\equiv i+m\pmod{k}$. Therefore, $v_h\not \in N(v_i)$ for all $h\equiv i\pmod{k}$.\\

\noindent\textbf{Case 2:} Suppose $i+kj+m>qk$, then $i+kj+m\equiv i+kj+m-1\pmod{qk+1}$, which is congruent to $i+m-1\pmod{k}$. Since $1\le m-1\le k-2$, it follows that $i\not\equiv i+m-1\pmod{k}$. Therefore, $v_h\not \in N(v_i)$ for all $h\equiv i\pmod{k}$.

\end{proof}

\begin{lemma}\label{lem:Gqkis2K2free}
For all $k\ge 4$ and $q\ge 1$, $\infcrit{q}{k}$ is $2K_2$-free.
\end{lemma}
\begin{proof}
Let $G=\infcrit{q}{k}$ and suppose $G$ is not $2K_2$-free. By symmetry, we may suppose without loss of generality that $v_0$ belongs to an induced $2K_2$. Let $v_\ell, x, y\in V(G)$ such that $\{v_0,v_\ell,x,y\}$ induces a $2K_2$ in $G$ such that $v_0\sim v_\ell$ and $x\sim y$. It is clear by the definition of $G$ that $N(v_0)=\{v_1,v_{qk}\}\cup V_2\cup V_3\cup \cdots \cup V_{k-1}$.  
Therefore, if $\{x,y\}\not\subseteq (V_0\setminus\{v_0,v_{qk}\})\cup V_1$, then $v_0$ will be adjacent to $x$ or $y$, contradicting our assumption. Further, since $x\sim y$, we must have $\{x,y\}\subseteq (V_0\setminus\{v_0,v_{qk}\})\cup (V_1\setminus\{v_1\})$. Since both $V_1\setminus\{v_1\}$ and $V_0\setminus\{v_0,v_{qk}\}$ are stable sets from Lemma~\ref{lem:VisstablesetsinGqkexceptV0}, we may assume without loss of generality that  $x\in V_0\setminus\{v_0,v_{qk}\}$ and $y\in V_1\setminus\{v_1\}$. There are now only three cases to consider.\\

\noindent\textbf{Case 1:} Suppose $v_\ell\in V_2\cup V_3\cup\cdots\cup V_{k-2}$. 
Let $\ell'=\ell\pmod{k}$. Since $k\ge 4$, $2\le \ell'\le k-2$. So, for $v_{hk}\in V_0$ where $\ell<hk\le qk$, we have that $hk=\ell+k(h-1)+(k-\ell')$. Thus, $v_{hk}\in N(v_\ell)$ by the definition of $G$. For $v_{hk}\in V_0$ where $qk<hk<\ell$, we have that $hk=\ell'+k(h-1)+(k-\ell'+1)\pmod{qk+1}$, so again $v_{hk}\in N(v_\ell)$. It follows that $V_0\subseteq N(v_\ell)$ and, thus, $v_\ell\sim x$. This contradicts our assumption on the induced $2K_2$ in $G$.\\

\noindent\textbf{Case 2:} Suppose $v_\ell\in V_{k-1}$. 
We show that $V_{k-1}\subseteq N(y)$. Let $h\in \{0,\ldots, qk\}$ such that $y=v_h$. We know that $h=1\pmod{k}$ since $y=v_h\in V_1$. Let $h'\in\{0,\ldots, qk\}$ such that $h'=k-1 \pmod{k}$.  
If $h< h'$, then since $h+kj+k-2\equiv k-1\pmod{k}$, it follows that $v_h\sim v_{h'}$. 
If $h'< h$, then since $h+kj+k-1\pmod{qk+1} \equiv k-1\pmod{k}$, it follows that $v_h\sim v_{h'}$. Thus, $V_{k-1}\subseteq N(y)$, and   $y\sim v_\ell$ which again contradicts our assumption on the induced $2K_2$ in $G$.\\

\noindent\textbf{Case 3:} Suppose $v_\ell\in \{v_{qk},v_1\}$.
If $v_\ell=v_1$, then since $1+jk+m\pmod{qk+1}=1+jk+m$ for all $0\le j\le q-1$ and $2\le m\le k-1$, we can obtain all $v_h$ such that $v_h\in V_{0}$ in the set $\{v_{1+(q-1)j+(k-1)}:j=0,1,\ldots,q-1\}\subseteq N(v_1)$. Also, we can obtain all $v_h$ such that $v_h\in V_{k-1}$ in the set $\{v_{1+(q-1)j+(k-2)}:j=0,1,\ldots,q-1\}\subseteq N(v_1)$. Thus, $x\sim v_1=v_\ell$ which contradicts our assumption on the induced $2K_2$ in $G$. Similarly, for $v_\ell=v_{qk}$, we get that $y\sim v_\ell$ which again contradicts our assumption on the induced $2K_2$ in $G$.\\

All three cases together show that $G$ is $2K_2$-free.
\end{proof}

\begin{lemma}\label{lem:GqkisK3UP1free}
For all $k\ge 4$ and $q\ge 1$, $\infcrit{q}{k}$ is $(K_3+P_1)$-free.
\end{lemma}
\begin{proof}
Let $G=\infcrit{q}{k}$ and suppose $G$ has an induced $K_3+P_1$. By the symmetry of $G$, we may assume that the isolated vertex in the $K_3+P_1$ is $v_0$.  Let $S=\{v_0,v_{i_1},v_{i_2},v_{i_3}\}$ be the subset of $V(G)$ that induces the $K_3+P_1$ with $v_0$ the isolated vertex. Since $v_0$ is anticomplete with $\{v_{i_1},v_{i_2},v_{i_3}\}$, it follows that $\{v_{i_1},v_{i_2},v_{i_3}\}\subset V_0\cup V_{1}\setminus \{v_{qk}\}$. However, since $V_0$ and $V_{1}\setminus \{v_{qk}\}$ stable sets by Lemma~\ref{lem:VisstablesetsinGqkexceptV0}, it follows that $G[V_0\cup V_{1}\setminus \{v_{qk}]\}$ is bipartite and therefore triangle-free. Therefore, $G$ is $(K_3+P_1)$-free.
\end{proof}

\begin{lemma}\label{lem:GqkisC5free}
For all $k\ge 5$ and $q\ge 1$, $\infcrit{q}{k}$ is $C_5$-free.
\end{lemma}
\begin{proof}
Let $k\ge 5$, $q\ge 1$, and $G=\infcrit{q}{k}$. Suppose by way of contradiction that $G$ contains an induced $C_5$. Without loss of generality suppose $v_0$ is on an induced $C_5$ in $G$ and let $\{v_0,v_{i_1},v_{i_2},v_{i_3},v_{i_4}\}$ induce a $C_5$ in $G$ such that $v_0\sim v_{i_1}$, $v_0\sim v_{i_4}$, and $v_{i_j}\sim v_{i_{j+1}}$ for $j=1,2,3$. By similar arguments to the proof of Lemma~\ref{lem:Gqkis2K2free}, we must have, without loss of generality, $v_{i_2}\in V_0\setminus\{v_0,v_{qk}\}$ and $v_{i_3}\in V_1\setminus\{v_1\}$. Further, since $N(v_0)=\{v_1,v_{qk}\}\cup V_2\cup V_3\cup \cdots \cup V_{k-1}$, we must have $\{v_{i_4},v_{i_5}\}\subseteq V_2\cup V_3\cup \cdots \cup V_{k-1}$. By a similar argument to that of Case 1 of the proof of Lemma~\ref{lem:Gqkis2K2free} and since $v_1\sim v_{qk}$, we must have $v_{i_4}\in V_{k-1}$ (or else $v_{i_2}\sim v_{i_4}$) and $v_{i_3}\in V_j$ for some $2\le j\le k-2$. Further, $j=k-2$, or else $v_{i_1}\sim v_{i_4}$. But now since $k\ge 5$, and therefore $k-2\ge 3$, we have $v_{i_3}\sim v_{i_1}$. This contradicts our assumption on the induced $C_5$. Therefore, $G$ is $C_5$-free.
\end{proof}

Note our proof of the following lemma is adapted directly from Theorem 2.4 of~\cite{Hoang2015} where it was shown that $\infcrit{q,4}$ (under the name $G_p$) is $4$-vertex-critical for all $q$.
\begin{lemma}\label{lem:Gqkiskcritical}
For all $k\ge 3$ and $q\ge 1$, $\infcrit{q}{k}$ is $(k+1)$-vertex-critical.
\end{lemma}
\begin{proof}
Let $G=\infcrit{q}{k}$. Note that for all $0\le i\le (q-1)k+1$, the set $\{v_{i+j}:j=0,1,\ldots,k-1\}$ induces a $K_{k}$. Thus, $\chi(G)\ge k$. Suppose $G$ is $k$-colourable, and without loss of generality, let vertices $v_i$ be assigned colour $i$ for $i=0,1,\ldots, k-1$. We now must have vertex $v_j$ be assigned colour $j\pmod{k}$ for all $k\le j\le qk-1$. But now $N(v_{qk})$ contains all $k$ colours. So $G$ is not $k$-colourable. From Lemma~\ref{lem:VisstablesetsinGqkexceptV0}, however, this partial colouring is valid and we may give vertex $v_{qk}$ colour $k$ to get that $\chi(G)=k+1$. Since $v_{qk}$ is the only vertex with colour $k$, it follows by the symmetry of $G$ that $G$ is $(k+1)$-vertex-critical.
\end{proof}

\begin{theorem}
There are infinitely many $k$-vertex-critical $(2P_2,K_3+P_1,C_5)$-free graphs for all $k\ge 6$.
\end{theorem}
\begin{proof}
The proof follows from Lemmas~\ref{lem:Gqkis2K2free}-\ref{lem:Gqkiskcritical}
\end{proof}

\begin{corollary}\label{cor:P5C5free}
There are infinitely many $k$-vertex-critical $(P_5,C_5)$-free graphs for all $k\ge 6$.
\end{corollary}

\section{State of $(P_5,H)$-free graphs when $H$ is of order 5}\label{sec:stateoftheart}

In this section we detail in a table whether there are only finitely many or infinitely many $k$-vertex-critical $(P_5,H)$-free graphs for all nonisomorphic graphs $H$ of order $5$. This is to make progress on the open problem raised in~\cite{KCameron2021}. For the names of graphs we are mostly following \url{https://www.graphclasses.org/smallgraphs.html#nodes5} with the notable exception that we use $+$ to denote disjoint union.

\begin{center} 
\renewcommand\arraystretch{1.2}
\def\c{0.2}
\begin{longtable}{|c|c|c|c|} 
\hline
\textbf{Graph} & \textbf{Graph name} & \textbf{Finite/Infinite} & \textbf{Reference}\\ \hline
\scalebox{\c}{ \begin{tikzpicture}
\GraphInit[vstyle=Classic]
\Vertex[L=\hbox{$0$},x=0.0cm,y=2.5cm]{v0}
\Vertex[L=\hbox{$1$},x=1.25cm,y=2.5cm]{v1}
\Vertex[L=\hbox{$2$},x=2.5cm,y=2.5cm]{v2}
\Vertex[L=\hbox{$3$},x=3.75cm,y=2.5cm]{v3}
\Vertex[L=\hbox{$4$},x=5.0cm,y=2.5cm]{v4}
\end{tikzpicture} } & $\overline{K_5}$ & finite & Ramsey's Theorem~\cite{Ramsey} \\ \hline
\scalebox{\c}{ \begin{tikzpicture}
\GraphInit[vstyle=Classic]
\Vertex[L=\hbox{$0$},x=0.0cm,y=2.5cm]{v0}
\Vertex[L=\hbox{$1$},x=1.25cm,y=2.5cm]{v1}
\Vertex[L=\hbox{$2$},x=2.5cm,y=2.5cm]{v2}
\Vertex[L=\hbox{$3$},x=3.75cm,y=2.5cm]{v3}
\Vertex[L=\hbox{$4$},x=5.0cm,y=2.5cm]{v4}
\Edge[](v0)(v1)
\end{tikzpicture} } & $P_2+3P_1$ & finite & \cite{CameronHoangSawada2022} \\ \hline
\scalebox{\c}{ \begin{tikzpicture}
\GraphInit[vstyle=Classic]
\Vertex[L=\hbox{$0$},x=0.0cm,y=2.5cm]{v0}
\Vertex[L=\hbox{$1$},x=1.25cm,y=2.5cm]{v1}
\Vertex[L=\hbox{$2$},x=2.5cm,y=2.5cm]{v2}
\Vertex[L=\hbox{$3$},x=3.75cm,y=2.5cm]{v3}
\Vertex[L=\hbox{$4$},x=5.0cm,y=2.5cm]{v4}
\Edge[](v0)(v1)
\Edge[](v1)(v2)
\end{tikzpicture} } & $P_3+2P_1$ & finite & \cite{AbuadasCameronHoangSawada2022} \\ \hline
\scalebox{\c}{ \begin{tikzpicture}
\GraphInit[vstyle=Classic]
\Vertex[L=\hbox{$0$},x=3.3333cm,y=0.0845cm]{v0}
\Vertex[L=\hbox{$1$},x=1.5866cm,y=5.0cm]{v1}
\Vertex[L=\hbox{$2$},x=0.0cm,y=0.0cm]{v2}
\Vertex[L=\hbox{$3$},x=5.0cm,y=1.704cm]{v3}
\Vertex[L=\hbox{$4$},x=1.6444cm,y=1.7315cm]{v4}
\Edge[](v0)(v4)
\Edge[](v1)(v4)
\Edge[](v2)(v4)
\end{tikzpicture} } & claw$+P_1$ & unknown & N/A \\ \hline
\scalebox{\c}{ \begin{tikzpicture}
\GraphInit[vstyle=Classic]
\Vertex[L=\hbox{$0$},x=3.3294cm,y=5.0cm]{v0}
\Vertex[L=\hbox{$1$},x=1.7268cm,y=0.0cm]{v1}
\Vertex[L=\hbox{$2$},x=5.0cm,y=1.6608cm]{v2}
\Vertex[L=\hbox{$3$},x=0.0cm,y=3.2983cm]{v3}
\Vertex[L=\hbox{$4$},x=2.5861cm,y=2.5358cm]{v4}
\Edge[](v0)(v4)
\Edge[](v1)(v4)
\Edge[](v2)(v4)
\Edge[](v3)(v4)
\end{tikzpicture} } & $K_{1,4}$ & finite & \cite{Kaminski2019} (see also \cite[Theorem~1]{LozinRautenbach2003}) \\ \hline
\scalebox{\c}{ \begin{tikzpicture}
\GraphInit[vstyle=Classic]
\Vertex[L=\hbox{$0$},x=0.0cm,y=5.0cm]{v0}
\Vertex[L=\hbox{$1$},x=4.5cm,y=0.0cm]{v1}
\Vertex[L=\hbox{$2$},x=2.5cm,y=2.5cm]{v2}
\Vertex[L=\hbox{$3$},x=0.0cm,y=.0cm]{v3}
\Vertex[L=\hbox{$4$},x=4.5cm,y=5.0cm]{v4}
\Edge[](v0)(v3)
\Edge[](v1)(v4)
\end{tikzpicture} } & $2K_2+P_1$ & infinite & Contains $2K_2$~\cite{Hoang2015} \\ \hline
\scalebox{\c}{ \begin{tikzpicture}
\GraphInit[vstyle=Classic]
\Vertex[L=\hbox{$0$},x=2.2746cm,y=1.584cm]{v0}
\Vertex[L=\hbox{$1$},x=0.0cm,y=5.0cm]{v1}
\Vertex[L=\hbox{$2$},x=5.0cm,y=2.4925cm]{v2}
\Vertex[L=\hbox{$3$},x=3.3333cm,y=0.0cm]{v3}
\Vertex[L=\hbox{$4$},x=1.0737cm,y=3.386cm]{v4}
\Edge[](v0)(v3)
\Edge[](v0)(v4)
\Edge[](v1)(v4)
\end{tikzpicture} } & $P_4+P_1$ & unknown & N/A \\ \hline
\scalebox{\c}{ \begin{tikzpicture}
\GraphInit[vstyle=Classic]
\Vertex[L=\hbox{$0$},x=4.4191cm,y=5.0cm]{v0}
\Vertex[L=\hbox{$1$},x=4.0083cm,y=4.3383cm]{v1}
\Vertex[L=\hbox{$2$},x=0.0cm,y=0.646cm]{v2}
\Vertex[L=\hbox{$3$},x=5.0cm,y=0.0cm]{v3}
\Vertex[L=\hbox{$4$},x=2.0283cm,y=2.5156cm]{v4}
\Edge[](v0)(v3)
\Edge[](v1)(v4)
\Edge[](v2)(v4)
\end{tikzpicture} } & $P_3+P_2$ & infinite & Contains $2K_2$~\cite{Hoang2015} \\ \hline
\scalebox{\c}{ \begin{tikzpicture}
\GraphInit[vstyle=Classic]
\Vertex[L=\hbox{$0$},x=2.3205cm,y=1.8117cm]{v0}
\Vertex[L=\hbox{$1$},x=3.6603cm,y=2.2706cm]{v1}
\Vertex[L=\hbox{$2$},x=5.0cm,y=2.2706cm]{v2}
\Vertex[L=\hbox{$3$},x=0.0cm,y=0.0cm]{v3}
\Vertex[L=\hbox{$4$},x=0.4202cm,y=5.0cm]{v4}
\Edge[](v0)(v3)
\Edge[](v0)(v4)
\Edge[](v3)(v4)
\end{tikzpicture} } & $K_3+2P_1$ & infinite & Contains $K_3+P_1$~\cite{KCameron2021,Hoang2015} \\ \hline
\scalebox{\c}{ \begin{tikzpicture}
\GraphInit[vstyle=Classic]
\Vertex[L=\hbox{$0$},x=1.419cm,y=3.5997cm]{v0}
\Vertex[L=\hbox{$1$},x=5.0cm,y=2.3527cm]{v1}
\Vertex[L=\hbox{$2$},x=2.817cm,y=0.0cm]{v2}
\Vertex[L=\hbox{$3$},x=0.0cm,y=5.0cm]{v3}
\Vertex[L=\hbox{$4$},x=2.987cm,y=2.0359cm]{v4}
\Edge[](v0)(v3)
\Edge[](v0)(v4)
\Edge[](v1)(v4)
\Edge[](v2)(v4)
\end{tikzpicture} } & chair & finite $k\le 5$, unknown $k\ge 6$ & \cite{HuangLi2023} \\ \hline
\scalebox{\c}{ \begin{tikzpicture}
\GraphInit[vstyle=Classic]
\Vertex[L=\hbox{$0$},x=0.0cm,y=4.8144cm]{v0}
\Vertex[L=\hbox{$1$},x=2.2989cm,y=0.0cm]{v1}
\Vertex[L=\hbox{$2$},x=5.0cm,y=3.1071cm]{v2}
\Vertex[L=\hbox{$3$},x=3.3333cm,y=5.0cm]{v3}
\Vertex[L=\hbox{$4$},x=1.9629cm,y=2.614cm]{v4}
\Edge[](v0)(v3)
\Edge[](v0)(v4)
\Edge[](v1)(v4)
\Edge[](v3)(v4)
\end{tikzpicture} } & co-dart & infinite & Contains $K_3+P_1$~\cite{KCameron2021,Hoang2015}\\ \hline
\scalebox{\c}{ \begin{tikzpicture}
\GraphInit[vstyle=Classic]
\Vertex[L=\hbox{$0$},x=3.8766cm,y=0.0cm]{v0}
\Vertex[L=\hbox{$1$},x=5.0cm,y=5.0cm]{v1}
\Vertex[L=\hbox{$2$},x=0.0cm,y=4.7704cm]{v2}
\Vertex[L=\hbox{$3$},x=1.6339cm,y=0.1528cm]{v3}
\Vertex[L=\hbox{$4$},x=2.6166cm,y=2.8152cm]{v4}
\Edge[](v0)(v3)
\Edge[](v0)(v4)
\Edge[](v1)(v4)
\Edge[](v2)(v4)
\Edge[](v3)(v4)
\end{tikzpicture} } & $\overline{\text{diamond}+P_1}$ & unknown & N/A \\ \hline
\scalebox{\c}{ \begin{tikzpicture}
\GraphInit[vstyle=Classic]
\Vertex[L=\hbox{$0$},x=3.3333cm,y=4.654cm]{v0}
\Vertex[L=\hbox{$1$},x=0.0cm,y=0.4646cm]{v1}
\Vertex[L=\hbox{$2$},x=5.0cm,y=2.5296cm]{v2}
\Vertex[L=\hbox{$3$},x=0.1013cm,y=5.0cm]{v3}
\Vertex[L=\hbox{$4$},x=3.3081cm,y=0.0cm]{v4}
\Edge[](v0)(v3)
\Edge[](v0)(v4)
\Edge[](v1)(v3)
\Edge[](v1)(v4)
\end{tikzpicture} } & $C_4+P_1$ & unknown & N/A \\ \hline
\scalebox{\c}{ \begin{tikzpicture}
\GraphInit[vstyle=Classic]
\Vertex[L=\hbox{$0$},x=3.7285cm,y=4.4694cm]{v0}
\Vertex[L=\hbox{$1$},x=0.0cm,y=2.916cm]{v1}
\Vertex[L=\hbox{$2$},x=5.0cm,y=0.0cm]{v2}
\Vertex[L=\hbox{$3$},x=0.5185cm,y=5.0cm]{v3}
\Vertex[L=\hbox{$4$},x=3.3307cm,y=2.2134cm]{v4}
\Edge[](v0)(v3)
\Edge[](v0)(v4)
\Edge[](v1)(v3)
\Edge[](v1)(v4)
\Edge[](v2)(v4)
\end{tikzpicture} } & banner & finite & \cite[Theorem~3(i)]{Brause2022}\\ \hline
\scalebox{\c}{ \begin{tikzpicture}
\GraphInit[vstyle=Classic]
\Vertex[L=\hbox{$0$},x=0.0cm,y=0.0cm]{v0}
\Vertex[L=\hbox{$1$},x=3.3333cm,y=5.0cm]{v1}
\Vertex[L=\hbox{$2$},x=5.0cm,y=2.52cm]{v2}
\Vertex[L=\hbox{$3$},x=1.0142cm,y=3.659cm]{v3}
\Vertex[L=\hbox{$4$},x=2.372cm,y=1.4212cm]{v4}
\Edge[](v0)(v3)
\Edge[](v0)(v4)
\Edge[](v1)(v3)
\Edge[](v1)(v4)
\Edge[](v3)(v4)
\end{tikzpicture} } & diamond$+P_1$ & infinite & Contains $K_3+P_1$~\cite{KCameron2021,Hoang2015} \\ \hline
\scalebox{\c}{ \begin{tikzpicture}
\GraphInit[vstyle=Classic]
\Vertex[L=\hbox{$0$},x=0.0cm,y=2.3886cm]{v0}
\Vertex[L=\hbox{$1$},x=3.5862cm,y=5.0cm]{v1}
\Vertex[L=\hbox{$2$},x=5.0cm,y=0.0cm]{v2}
\Vertex[L=\hbox{$3$},x=2.9776cm,y=3.3948cm]{v3}
\Vertex[L=\hbox{$4$},x=3.3115cm,y=1.6559cm]{v4}
\Edge[](v0)(v3)
\Edge[](v0)(v4)
\Edge[](v1)(v3)
\Edge[](v2)(v4)
\Edge[](v3)(v4)
\end{tikzpicture} } & bull & finite $k=5$, unknown $k\ge 6$ & \cite{HuangLiXia2023} \\ \hline
\scalebox{\c}{ \begin{tikzpicture}
\GraphInit[vstyle=Classic]
\Vertex[L=\hbox{$0$},x=4.7359cm,y=0.0cm]{v0}
\Vertex[L=\hbox{$1$},x=3.9642cm,y=5.0cm]{v1}
\Vertex[L=\hbox{$2$},x=0.0cm,y=1.2356cm]{v2}
\Vertex[L=\hbox{$3$},x=5.0cm,y=2.6599cm]{v3}
\Vertex[L=\hbox{$4$},x=2.8129cm,y=2.0353cm]{v4}
\Edge[](v0)(v3)
\Edge[](v0)(v4)
\Edge[](v1)(v3)
\Edge[](v1)(v4)
\Edge[](v2)(v4)
\Edge[](v3)(v4)
\end{tikzpicture} } & dart & unknown & N/A \\ \hline
\scalebox{\c}{ \begin{tikzpicture}
\GraphInit[vstyle=Classic]
\Vertex[L=\hbox{$0$},x=0cm,y=0.0cm]{v0}
\Vertex[L=\hbox{$1$},x=3cm,y=0.0cm]{v1}
\Vertex[L=\hbox{$2$},x=6cm,y=0.0cm]{v2}
\Vertex[L=\hbox{$3$},x=1.5cm,y=3cm]{v3}
\Vertex[L=\hbox{$4$},x=4.5cm,y=3cm]{v4}
\Edge[](v0)(v3)
\Edge[](v0)(v4)
\Edge[](v1)(v3)
\Edge[](v1)(v4)
\Edge[](v2)(v3)
\Edge[](v2)(v4)
\end{tikzpicture} } & $K_{2,3}$& finite & \cite{Kaminski2019} \\ \hline
\scalebox{\c}{ \begin{tikzpicture}
\GraphInit[vstyle=Classic]
\Vertex[L=\hbox{$0$},x=0cm,y=0.0cm]{v0}
\Vertex[L=\hbox{$1$},x=3cm,y=0.0cm]{v1}
\Vertex[L=\hbox{$2$},x=6cm,y=0.0cm]{v2}
\Vertex[L=\hbox{$3$},x=1.5cm,y=3cm]{v3}
\Vertex[L=\hbox{$4$},x=4.5cm,y=3cm]{v4}
\Edge[](v0)(v3)
\Edge[](v0)(v4)
\Edge[](v1)(v3)
\Edge[](v1)(v4)
\Edge[](v2)(v3)
\Edge[](v2)(v4)
\Edge[](v3)(v4)
\end{tikzpicture} } & $\overline{K_3+2P_2}$ & unknown & N/A \\ \hline
\scalebox{\c}{ \begin{tikzpicture}
\GraphInit[vstyle=Classic]
\Vertex[L=\hbox{$0$},x=1.1907cm,y=1.1719cm]{v0}
\Vertex[L=\hbox{$1$},x=3.8728cm,y=3.8724cm]{v1}
\Vertex[L=\hbox{$2$},x=0.0cm,y=0.0cm]{v2}
\Vertex[L=\hbox{$3$},x=5.0cm,y=5.0cm]{v3}
\Vertex[L=\hbox{$4$},x=2.5913cm,y=2.5692cm]{v4}
\Edge[](v0)(v2)
\Edge[](v0)(v4)
\Edge[](v1)(v3)
\Edge[](v1)(v4)
\end{tikzpicture} } & $P_5$ & infinite & Contains $2K_2$~\cite{Hoang2015} \\ \hline
\scalebox{\c}{ \begin{tikzpicture}
\GraphInit[vstyle=Classic]
\Vertex[L=\hbox{$0$},x=0.0cm,y=2.721cm]{v0}
\Vertex[L=\hbox{$1$},x=3.9177cm,y=0.0cm]{v1}
\Vertex[L=\hbox{$2$},x=3.1523cm,y=4.5804cm]{v2}
\Vertex[L=\hbox{$3$},x=5.0cm,y=5.0cm]{v3}
\Vertex[L=\hbox{$4$},x=3.047cm,y=0.1986cm]{v4}
\Edge[](v0)(v2)
\Edge[](v0)(v4)
\Edge[](v1)(v3)
\Edge[](v2)(v4)
\end{tikzpicture} } & $K_3+P_2$ & infinite & Contains $2K_2$~\cite{Hoang2015} \\ \hline
\scalebox{\c}{ \begin{tikzpicture}
\GraphInit[vstyle=Classic]
\Vertex[L=\hbox{$0$},x=4.3651cm,y=0.0cm]{v0}
\Vertex[L=\hbox{$1$},x=1.5062cm,y=3.6722cm]{v1}
\Vertex[L=\hbox{$2$},x=5.0cm,y=1.7455cm]{v2}
\Vertex[L=\hbox{$3$},x=0.0cm,y=5.0cm]{v3}
\Vertex[L=\hbox{$4$},x=3.3082cm,y=2.0683cm]{v4}
\Edge[](v0)(v2)
\Edge[](v0)(v4)
\Edge[](v1)(v3)
\Edge[](v1)(v4)
\Edge[](v2)(v4)
\end{tikzpicture} } & co-banner & infinite & Contains $2K_2$~\cite{Hoang2015} (see also~\cite{Brause2022}) \\ \hline
\scalebox{\c}{ \begin{tikzpicture}
\GraphInit[vstyle=Classic]
\Vertex[L=\hbox{$0$},x=4.5353cm,y=5.0cm]{v0}
\Vertex[L=\hbox{$1$},x=0.0cm,y=2.9292cm]{v1}
\Vertex[L=\hbox{$2$},x=5.0cm,y=2.0041cm]{v2}
\Vertex[L=\hbox{$3$},x=0.5146cm,y=0.0cm]{v3}
\Vertex[L=\hbox{$4$},x=2.5801cm,y=2.5676cm]{v4}
\Edge[](v0)(v2)
\Edge[](v0)(v4)
\Edge[](v1)(v3)
\Edge[](v1)(v4)
\Edge[](v2)(v4)
\Edge[](v3)(v4)
\end{tikzpicture} } & butterfly & infinite & Contains $2K_2$~\cite{Hoang2015} \\ \hline
\scalebox{\c}{ \begin{tikzpicture}
\GraphInit[vstyle=Classic]
\Vertex[L=\hbox{$0$},x=0.0cm,y=2.0355cm]{v0}
\Vertex[L=\hbox{$1$},x=5.0cm,y=1.3449cm]{v1}
\Vertex[L=\hbox{$2$},x=1.1563cm,y=5.0cm]{v2}
\Vertex[L=\hbox{$3$},x=2.3237cm,y=0.0cm]{v3}
\Vertex[L=\hbox{$4$},x=4.114cm,y=4.5283cm]{v4}
\Edge[](v0)(v2)
\Edge[](v0)(v3)
\Edge[](v1)(v3)
\Edge[](v1)(v4)
\Edge[](v2)(v4)
\end{tikzpicture} } & $C_5$ & finite $k\le 5$, infinite $k\ge 6$ & \cite{Hoang2015}, Corollary~\ref{cor:P5C5free} \\ \hline
\scalebox{\c}{ \begin{tikzpicture}
\GraphInit[vstyle=Classic]
\Vertex[L=\hbox{$0$},x=3.3673cm,y=3.313cm]{v0}
\Vertex[L=\hbox{$1$},x=1.2924cm,y=0.0cm]{v1}
\Vertex[L=\hbox{$2$},x=0.4987cm,y=5.0cm]{v2}
\Vertex[L=\hbox{$3$},x=5.0cm,y=0.8231cm]{v3}
\Vertex[L=\hbox{$4$},x=0.0cm,y=2.573cm]{v4}
\Edge[](v0)(v2)
\Edge[](v0)(v3)
\Edge[](v0)(v4)
\Edge[](v1)(v3)
\Edge[](v1)(v4)
\Edge[](v2)(v4)
\end{tikzpicture} } & $\overline{P_5}$ & finite & \cite{Dhaliwal2017} \\ \hline
\scalebox{\c}{ \begin{tikzpicture}
\GraphInit[vstyle=Classic]
\Vertex[L=\hbox{$0$},x=3.4786cm,y=5.0cm]{v0}
\Vertex[L=\hbox{$1$},x=0.0cm,y=0.0cm]{v1}
\Vertex[L=\hbox{$2$},x=5.0cm,y=1.7942cm]{v2}
\Vertex[L=\hbox{$3$},x=0.9636cm,y=4.0753cm]{v3}
\Vertex[L=\hbox{$4$},x=2.5013cm,y=1.0642cm]{v4}
\Edge[](v0)(v2)
\Edge[](v0)(v3)
\Edge[](v0)(v4)
\Edge[](v1)(v3)
\Edge[](v1)(v4)
\Edge[](v2)(v4)
\Edge[](v3)(v4)
\end{tikzpicture} } & gem & finite & \cite{CaiGoedgebeurHuang2021} (see also~\cite{CameronHoang2023}) \\ \hline
\scalebox{\c}{ \begin{tikzpicture}
\GraphInit[vstyle=Classic]
\Vertex[L=\hbox{$0$},x=0.0cm,y=3.3251cm]{v0}
\Vertex[L=\hbox{$1$},x=5.0cm,y=0.0cm]{v1}
\Vertex[L=\hbox{$2$},x=3.3091cm,y=3.7807cm]{v2}
\Vertex[L=\hbox{$3$},x=0.1465cm,y=5.0cm]{v3}
\Vertex[L=\hbox{$4$},x=3.4713cm,y=1.8121cm]{v4}
\Edge[](v0)(v2)
\Edge[](v0)(v3)
\Edge[](v0)(v4)
\Edge[](v1)(v4)
\Edge[](v2)(v3)
\Edge[](v2)(v4)
\end{tikzpicture} } & kite & infinite & Contains $K_3+P_1$~\cite{KCameron2021,Hoang2015} \\ \hline
\scalebox{\c}{ \begin{tikzpicture}
\GraphInit[vstyle=Classic]
\Vertex[L=\hbox{$0$},x=2.8604cm,y=0.0cm]{v0}
\Vertex[L=\hbox{$1$},x=5.0cm,y=2.5003cm]{v1}
\Vertex[L=\hbox{$2$},x=0.4766cm,y=5.0cm]{v2}
\Vertex[L=\hbox{$3$},x=0.0cm,y=0.7264cm]{v3}
\Vertex[L=\hbox{$4$},x=3.3333cm,y=4.2748cm]{v4}
\Edge[](v0)(v2)
\Edge[](v0)(v3)
\Edge[](v0)(v4)
\Edge[](v2)(v3)
\Edge[](v2)(v4)
\Edge[](v3)(v4)
\end{tikzpicture} } & $K_4+P_1$ & infinite & Contains $K_3+P_1$~\cite{KCameron2021,Hoang2015} \\ \hline
\scalebox{\c}{ \begin{tikzpicture}
\GraphInit[vstyle=Classic]
\Vertex[L=\hbox{$0$},x=5.0cm,y=1.5577cm]{v0}
\Vertex[L=\hbox{$1$},x=0.0cm,y=5.0cm]{v1}
\Vertex[L=\hbox{$2$},x=3.6959cm,y=0.0cm]{v2}
\Vertex[L=\hbox{$3$},x=4.368cm,y=4.2227cm]{v3}
\Vertex[L=\hbox{$4$},x=2.5327cm,y=3.1431cm]{v4}
\Edge[](v0)(v2)
\Edge[](v0)(v3)
\Edge[](v0)(v4)
\Edge[](v1)(v4)
\Edge[](v2)(v3)
\Edge[](v2)(v4)
\Edge[](v3)(v4)
\end{tikzpicture} } & $\overline{\text{claw}+K_1}$ & infinite & Contains $K_3+P_1$~\cite{KCameron2021,Hoang2015} \\ \hline
\scalebox{\c}{ \begin{tikzpicture}
\GraphInit[vstyle=Classic]
\Vertex[L=\hbox{$0$},x=1.6274cm,y=0.0cm]{v0}
\Vertex[L=\hbox{$1$},x=5.0cm,y=5.0cm]{v1}
\Vertex[L=\hbox{$2$},x=0.0cm,y=1.9492cm]{v2}
\Vertex[L=\hbox{$3$},x=2.2375cm,y=3.5065cm]{v3}
\Vertex[L=\hbox{$4$},x=3.7159cm,y=1.9038cm]{v4}
\Edge[](v0)(v2)
\Edge[](v0)(v3)
\Edge[](v0)(v4)
\Edge[](v1)(v3)
\Edge[](v1)(v4)
\Edge[](v2)(v3)
\Edge[](v2)(v4)
\Edge[](v3)(v4)
\end{tikzpicture} } & $\overline{P_3+2P_1}$ & unknown & N/A \\ \hline
\scalebox{\c}{ \begin{tikzpicture}
\GraphInit[vstyle=Classic]
\Vertex[L=\hbox{$0$},x=3.5002cm,y=3.4431cm]{v0}
\Vertex[L=\hbox{$1$},x=1.8621cm,y=0.8079cm]{v1}
\Vertex[L=\hbox{$2$},x=1.3384cm,y=5.0cm]{v2}
\Vertex[L=\hbox{$3$},x=5.0cm,y=0.0cm]{v3}
\Vertex[L=\hbox{$4$},x=0.0cm,y=2.8691cm]{v4}
\Edge[](v0)(v2)
\Edge[](v0)(v3)
\Edge[](v0)(v4)
\Edge[](v1)(v2)
\Edge[](v1)(v3)
\Edge[](v1)(v4)
\Edge[](v2)(v4)
\end{tikzpicture} } & paraglider or $\overline{P_3+P_2}$ & finite & \cite{CaiGoedgebeurHuang2021} \\ \hline
\scalebox{\c}{ \begin{tikzpicture}
\GraphInit[vstyle=Classic]
\Vertex[L=\hbox{$0$},x=5.0cm,y=0.3734cm]{v0}
\Vertex[L=\hbox{$1$},x=0.0cm,y=4.4903cm]{v1}
\Vertex[L=\hbox{$2$},x=4.6651cm,y=5.0cm]{v2}
\Vertex[L=\hbox{$3$},x=0.4475cm,y=0.0cm]{v3}
\Vertex[L=\hbox{$4$},x=2.4963cm,y=2.4758cm]{v4}
\Edge[](v0)(v2)
\Edge[](v0)(v3)
\Edge[](v0)(v4)
\Edge[](v1)(v2)
\Edge[](v1)(v3)
\Edge[](v1)(v4)
\Edge[](v2)(v4)
\Edge[](v3)(v4)
\end{tikzpicture} } & $W_4$ & unknown & N/A \\ \hline
\scalebox{\c}{ \begin{tikzpicture}
\GraphInit[vstyle=Classic]
\Vertex[L=\hbox{$0$},x=0.0cm,y=1.4442cm]{v0}
Vertex[NoLabel,[L=\hbox{$1$},x=5.0cm,y=4.7619cm]{v1}
Vertex[NoLabel,[L=\hbox{$2$},x=2.0998cm,y=5.0cm]{v2}
Vertex[NoLabel,[L=\hbox{$3$},x=2.4654cm,y=0.0cm]{v3}
Vertex[NoLabel,[L=\hbox{$4$},x=5.0cm,y=1.7609cm]{v4}
\Edge[](v0)(v2)
\Edge[](v0)(v3)
\Edge[](v0)(v4)
\Edge[](v1)(v2)
\Edge[](v1)(v3)
\Edge[](v1)(v4)
\Edge[](v2)(v3)
\Edge[](v2)(v4)
\Edge[](v3)(v4)
\end{tikzpicture} } & $K_5-e$ & unknown & N/A \\ \hline
\scalebox{\c}{ \begin{tikzpicture}
\GraphInit[vstyle=Classic]
Vertex[NoLabel,[L=\hbox{$0$},x=1.9646cm,y=0.0cm]{v0}
Vertex[NoLabel,[L=\hbox{$1$},x=4.8877cm,y=0.6856cm]{v1}
Vertex[NoLabel,[L=\hbox{$2$},x=0.0cm,y=2.6875cm]{v2}
Vertex[NoLabel,[L=\hbox{$3$},x=5.0cm,y=4.3217cm]{v3}
Vertex[NoLabel,[L=\hbox{$4$},x=2.0868cm,y=5.0cm]{v4}
\Edge[](v0)(v1)
\Edge[](v0)(v2)
\Edge[](v0)(v3)
\Edge[](v0)(v4)
\Edge[](v1)(v2)
\Edge[](v1)(v3)
\Edge[](v1)(v4)
\Edge[](v2)(v3)
\Edge[](v2)(v4)
\Edge[](v3)(v4)
\end{tikzpicture} } & $K_5$ & infinite $k=5$, unknown $k\ge 6$ & \cite{Hoang2015} \\ \hline
\caption{The state-of-the-art for $H$ of order $5$.}\label{tab:order5suvery}
\end{longtable}
\end{center}

\bibliographystyle{abbrv}
\bibliography{refs}

\end{document}